\documentclass{amsart}

\usepackage{amsmath,amsfonts,amssymb,amsthm,enumerate,tikz-cd,color}

\def\Z{\mathbb Z}
\def\Q{\mathbb Q}
\def\C{\mathbb C}

\newcommand{\an}{\mathrm{an}}
\newcommand{\divv}{\mathrm{div}}
\newcommand{\Div}{\mathrm{Div}}

\newcommand{\rk}{\mathrm{rk}}
\newcommand{\ord}{\mathrm{ord}}

\newcommand{\val}{\mathrm{val}}

\newtheorem{conjecture}{Conjecture}
\newtheorem{theorem}{Theorem}
\newtheorem{lemma}{Lemma}
\newtheorem{proposition}{Proposition}

\theoremstyle{definition}

\newtheorem{remark}{Remark}

\newcommand{\Gal}{\mathrm{Gal}}
\newcommand{\nr}{\mathrm{nr}}
\newcommand{\F}{\mathbb{F}}

\newcommand{\dlog}{\mathrm{dlog}}
\newcommand{\tors}{\mathrm{tors}}
\newcommand{\Jac}{\mathrm{Jac}}
\newcommand{\NP}{\mathrm{NP}}
\newcommand{\R}{\mathbb{R}}

\begin{document}
\title{$p$-adic approaches to unlikely intersections}
\author{Netan Dogra}
\maketitle
\begin{abstract}
In this article we explain the Buium--Coleman approach to the Manin--Mumford conjecture, and outline its generalisations. As an illustration, we give a $p$-adic proof of a theorem of Bombieri, Masser and Zannier on curves in tori.
\end{abstract}
\section{Introduction}
The Mordell conjecture, proved by Faltings \cite{faltings83}, states that if $X$ is a smooth projective geometrically irreducible curve of genus $g>1$ over a number field $K$, then $X(K)$ is finite. The Mordell--Lang conjecture, also proved by Faltings \cite{faltings:general}, may be viewed as saying that the Mordell--Weil theorem is the only arithmetic input needed. More precisely, it states that if $X$ is a subvariety of a complex abelian variety $A$, and $\Gamma <A(\mathbb{C})$ is a subgroup with $\dim _{\Q }\Gamma \otimes _{\Z }\Q $ finite, then the Zariski closure of $X(\C )\cap \Gamma $ is a finite union of translates of abelian subvarieties of $A$.

Later, Richard Pink fitted the Mordell--Lang conjecture into a broader framework that also related to an earlier conjecture of Boris Zilber \cite{pink} \cite{zilber}. This broad conjecture, known as the Zilber--Pink conjecture, includes as special cases the Mordell--Lang conjecture and the Andr\'e--Oort conjecture on CM points in Shimura varieties. In this article we will only discuss this conjecture in the case of subvarieties of semiabelian varieties. Given a semiabelian variety $G/\mathbb{C}$, and a positive integer $m$, we write $G^{[m]}$ to denote the union of complex points of all codimension $\geq m$ subgroups of $G$.
\begin{conjecture}[Pink]\label{conj:pink}
Let $X/\mathbb{C}$ be a subvariety of dimension $d$ of a semiabelian variety $G$, such that $X$ is not contained in a proper abelian subgroup of $G$. Then $X\cap G^{[d+1]}$ is not Zariski dense.
\end{conjecture}
\begin{remark}
The above conjecture is not what is usually referred to as the Zilber--Pink conjecture. However (although the Zilber--Pink conjecture is still wide open) it has been proved by Barroero and Dill \cite{BD25} that the conjecture above is \textit{equivalent} to the usual form of the Zilber--Pink conjecture. We refer the reader to \cite{pila} for a survey of the Zilber--Pink conjecture and known results.
\end{remark}
The relation between Mordell--Lang and Zilber--Pink comes from the apparently unpromising observation that, to show that a subset $T$ is not dense in a topological space $S$, it is enough to show that $T^n$ is not dense in $S^n$. A rational point in $X$ will not typically lie in a proper abelian subvariety of its Jacobian $J$, and hence the subset $X(\Q )\subset J(\mathbb{C})$ is a priori not something that can be attacked using Zilber--Pink. However, if $\rk J(\Q )=r$ and $n>0$ then every point of $X(\Q )^{r+n}$ lies on a codimension $ng$ abelian subvariety of $J(\C )^{r+n}$, and hence, once $n>\frac{r}{g-1}$, is contained in the unlikely intersection $X^{r+n}(\mathbb{C})\cap (J(\mathbb{C})^{r+n} )^{[ng]}$ (here for simplicity we assume that $X$ is furthermore not contained in a translate of a proper abelian subvariety of $G$, so that $X^{r+n}$ generates $J^{r+n}$).

Although this application of the Zilber--Pink conjecture proves a result which is already known, there are other cases of the Zilber--Pink conjecture which imply Diophantine results currently out of reach. For example by Stoll \cite{stoll:uniform}, Zilber--Pink implies results on the non-Zariski density of the set of `low-rank' points on $X\times _S \ldots \times _S X$, where $X\to S$ is a family of subvarieties of abelian varieties (see \cite{D22} for more discussion of this example).

For applications to rational points, it is evidently overkill to work over $\mathbb{C}$ in conjecture \ref{conj:pink}. However one could clearly formulate the same conjecture over any subfield $K$ of $\mathbb{C}$: one can either define a subset $G^{[m]}$ of $G(K)$ as above, or choose an embedding into $\mathbb{C}$ and conjecture that $X(K)\cap G^{[d+1]}$ is not Zariski dense in $X_K$. We are particularly interested in examining this conjecture over a field $K$ which is an extension of $\Q _p $. It turns out to be useful to separate into three different contexts: 
\begin{enumerate}
\item $K$ is a finite extension of $\Q _p $. 
\item $K$ is a finite extension of $\Q _p ^{\nr }$ (the maximal unramified extension of $\Q _p $).
\item $K=\overline{\Q }_p $. 
\end{enumerate}
It is straightforward to show that proving conjecture \ref{conj:pink} over $\overline{\Q }_p$ is equivalent to proving it over $\mathbb{C}$. The Buium--Coleman method may be described as being a strategy to prove cases of Zilber--Pink as follows: first, reduce case (3) to (2) using Coleman's method for bounding ramification of common zeroes of $p$-adic abelian integrals \cite{coleman:ramified}.  Second, use Buium's theory of $\delta $-geometry and $\delta $-characters to reduce case (2) to case (1) \cite{buium:96,buium:95,buium:97}. Finally, solve case (1) using a version of the Chabauty--Coleman method \cite{coleman:chabauty}. This is roughly the strategy of the Buium--Coleman proof of Manin--Mumford for curves inside their Jacobian, and has recently been employed by the author, with Pandit, to give a new proof of the Mordell--Lang conjecture for curves, which is quantitative when the rank is less than the genus \cite{DP}. In the case when the rank is greater than the genus, the Chabauty--Coleman method is not available, and the Buium--Coleman method instead reduces the Mordell--Lang conjecture to the Mordell--Lang conjecture (such a reduction was obtained by somewhat different methods by Raynaud \cite{Ray83}).

As an illustration of this circle of ideas, in this article we give a `Buium--Coleman' proof of the following theorem first proved by Bombieri, Masser and Zannier \cite{BMZ99}.
\begin{theorem}\label{thm:main}
Let $X$ be a curve over a number field $K$ and $f_1 ,f_2 ,f_3 \in K(X)^\times $ be functions which are linearly independent in $(K(X)^\times /K^\times )\otimes \Q $. Then the set of points $x\in X(\overline{K})$ such that $(f_i (x))$ lies in a rank one subgroup of $\mathbb{G}_m ^3 $ is finite.
\end{theorem}
In the language of Zilber--Pink, this is saying that $X(\overline{K})\cap (\mathbb{G}_m ^3 )^{[2]}$ is finite. In fact Bombieri, Masser and Zannier proved this more generally for a curve $C$ inside $\mathbb{G}_m ^n $ when $n\geq 2$. An alternative proof of the theorem was given by Capuano, Masser, Pila and Zannier \cite{CMPZ} using the Pila--Zannier strategy \cite{PZ}.
This theorem was proved by Maurin \cite{maurin} with the weaker assumption that the $f_i$ are linearly independent in $K(X)^\times $. We refer the reader to \cite{Zann} for a survey of unlikely intersection problems in tori.

\subsection*{Acknowledgements}
I am very grateful to Gabriel Dill and the anonymous referees for their corrections. The ideas in this paper originated from discussions with Sudip Pandit and Arnab Saha. I would also like to thank Chris Daw, Martin Orr, Zerui Tan and David Urbanik for teaching me about unlikely intersections. This research was support by Royal Society University Research Fellowship URF$\backslash $R1$\backslash $201215.

\section{$p$-adic integrals and the Buium--Coleman approach to Manin--Mumford}
There are multiple proofs of the Manin--Mumford conjecture (see \cite{tzermias} for a survey). In this section we want to explain the Buium--Coleman approach via $p$-adic integration \cite{buium:96}, \cite{coleman:ramified}. We will deviate very slightly from the original proofs. We refer to \cite{FvdP} for background on rigid analysis. Let $X$ be a smooth projective variety over $\overline{\Z }_p $. Let $X^{\an }$ denote the rigid analytification of $X_{\overline{\Q }_p }$. If $\overline{z}$ is an $\overline{\F }_p$ point of $X$, we denote by $]\overline{z}[$ the tube of $\{ \overline{z} \}$ in the sense of Berthelot \cite{LS}. This is a rigid analytic space, isomorphic to an open polydisk, whose $\mathbb{C}_p $ points are the preimage of $\overline{z}$ under the reduction map
\[
X^{\mathrm{an}}(\mathbb{C}_p )\to X(\overline{\F }_p ).
\]

There are several important features of open polydiscs $B$ it will be useful to highlight:
\begin{itemize}
\item one can integrate on $B$: for every closed one-form $\omega $ in $\mathcal{O}(B)$ there is a $g$ in $\mathcal{O}(B)$ such that $dg=\omega $. By contrast if $D$ is a closed polydisc then $\Omega (D)/d\mathcal{O}(D)$ is infinite dimensional.
\item The open polydisc is non-Noetherian. For example, in one dimension rigid analytic functions on it can have infinitely many zeroes. In higher dimensions, the common zeroes of a set of rigid analytic functions on an open polydisc will be a rigid analytic space with (potentially) infinitely many irreducible components.
\item For any \textit{finitely ramified} extension $K$ of $\Q _p $, the $K$ points of an open polydisc will be contained in a closed disc. Hence, for example, in one dimension a rigid analytic function on an open disc can have only finitely many zeroes in $K$. This feature will be discussed more in section \ref{sec:C}.
\end{itemize}

The $p$-adic logarithm
\[
\log :\overline{\Z }_p ^\times \to \overline{\Q }_p 
\]
is the unique homomorphism which is given, on the residue disc $]1[$, by the rigid analytic function
\[
1+T\mapsto \sum (-1)^{i+1}\frac{T^i}{i}.
\]

Similarly if $A/\overline{\Z }_p$ is a semiabelian variety, and $\omega $ is a global differential on $A$, then $\int _0 \omega $ can be defined as the unique homomorphism
\[
A(\overline{\Z }_p )\to \overline{\Q }_p .
\]
which agrees with the formal integral of $\omega $ on the residue disc at $0$ \cite{coleman:torsion} \cite{zarhin}. In particular, for all torsion points $P$ on $A$,
\[
\int ^P _0 \omega =0.
\]
We will refer to this as the \textit{Coleman integral} of $\omega $. Note that the function $\log $ above is a special case of this construction when $A=\mathbb{G}_m$ and $\omega =\frac{dT}{T}$. If $X$ is a curve mapped into its Jacobian via sending a point $b$ to the origin, then by functoriality we deduce that $X(\overline{\Q }_p )\cap J_{\mathrm{tors}}$ is contained in the common zeroes of $\int _b \omega $ for $\omega \in H^0 (X,\Omega )$.
The Buium--Coleman approach to the Manin--Mumford conjecture for a curve $X$ is essentially to show that the common zeroes of $\int _b \omega$,
for $\omega $ in $H^0 (X,\Omega )$, are finite.

If $\int _b \omega $ were a rigid analytic function on a one dimensional Noetherian rigid analytic space, then we could conclude by observing that it has only finitely many zeroes. However, this is not the case: it is only locally analytic, i.e.  the restriction of $\int _b \omega $ to each residue disc is a rigid analytic function with differential $\omega $. The residue discs are not Noetherian rigid analytic spaces, and in fact $\int _b \omega $ can have infinitely many zeroes. There are two essentially distinct ways that this happens.

\begin{enumerate}
\item Since the open discs $]\overline{z}[$,  for $\overline{z}\in X(\overline{\F }_p )$ are non-Noetherian, on each residue disk the rigid analytic function $(\int _b \omega )|_{]\overline{z}[}$ can have infinitely many zeroes. In fact one can show that they do have infinitely many zeroes (see Coleman \cite{coleman:ramified}). Note that this is not an issue if one is only interested in $\Q _p ^{\mathrm{nr}}$-points (or more generally points of bounded ramification), because the $\Q _p ^{\mathrm{nr}}$ points all lie on the closed disc of radius $\frac{1}{p}$ inside $]\overline{z}[$. 
\item Even restricting to $\Q _p^{\mathrm{nr}}$-points, $\int _b \omega $ is only rigid analytic when restricted to a residue disc $]\overline{z}[$. Since there are infinitely many residue discs, any argument which only uses the fact that the functions $\int _b \omega $ are locally analytic cannot prove Manin--Mumford, and one needs a global way to understand their behaviour.
\end{enumerate}
An archetypal example of this phenomenon is the logarithm map $\log $ on $\mathbb{G}_m$. The zeroes of $\log $ are exactly the roots of unity in $\overline{\Q }_p $. On each residue disc, $\log $ has infinitely many zeroes, but only finitely many of bounded ramification. On the other hand on the whole of $\mathbb{G}_m $ it also has infinitely many unramified zeroes. The first issue is about ramification, and is dealt with using the notion of the \textit{Coleman expansion}. The second issue can be dealt with using Buium's theory of delta characters \cite{buium:95} \cite{buium:96}, which replaces locally analytic functions on a variety with (global) rigid analytic functions on a larger dimensional space.

\section{The Coleman step: bounding ramification of anomalous points}\label{sec:C}
In this section we will give some of the ideas behind Coleman's result \cite{coleman:ramified} on ramification of torsion points, in a substantially simpler setting. The basic idea is to control the ramification of zeroes of Coleman integrals by finding a geometric interpretation of their valuations. Here $\val :\C _p \to \R \cup \{ \infty \}$ is the valuation normalised so that $\val (p)=1$, and the valuations of the zeroes are their valuations with respect to a suitably chosen integral parameter. Studying ramified points can be reduced to studying points of non-integral valuation, by the following elementary observation.
\begin{lemma}
Let $X$ be a curve over $\Z _p ^{\mathrm{nr}}$. A point $z\in X(\overline{\Z }_p )$ is ramified if and only if there is a point $z_0 \in X(\Z _p ^{\mathrm{nr}})$, and an integral parameter $T$ at $z_0 $, such that $\val (T(z))$ is not in $\Z $.
\end{lemma}
\begin{proof}
See e.g. \cite[3.2]{DP}.
\end{proof}

Given $F=\sum a_n t^n \in \C _p [\! [t]\! ]$, the Newton polygon $\mathrm{NP}(F)$ is the highest convex polygon lying above the points $(n,\val (a_n ))$. We say that $\lambda $ is a negative slope of $F$ if there is a line segment in $\NP (F)$ of gradient $-\lambda $. The \textit{length} of the slope is the length of the projection of the corresponding line segment onto the $x$-axis. The relation with zeroes of $F$ is the following (see e.g. \cite[IV.4]{koblitz}).
\begin{proposition}
The valuations of the zeroes of $F$ are the negative slopes of $\mathrm{NP}(F)$. Moreover the number of zeroes within given a valuation is equal to the length of the corresponding slope.
\end{proposition}
For example, the function $\log (1+t)$ is easily seen to have a Newton polygon with end points $(p^i ,-i)$ for $i\in \Z _{\geq 0}$. Hence the negative slopes are $\frac{1}{p^i-p^{i-1}}$ for $i>0$, and their lengths are $p^i -p^{i-1}$. This corresponds to the fact that $\val (\zeta _{p^i}-1)=\frac{1}{p^i-p^{i-1}}$.

Let 
\[
X\stackrel{(f_1 ,\ldots ,f_r )}{\longrightarrow }\mathbb{G}_m ^r
\]
be an irreducible closed curve over $\mathbb{Z}_p ^{\mathrm{nr}}$, whose normalisation $\widetilde{X}\to X$ has projective closure $C$ of genus $g$, with $\widetilde{X}\subset C$ being the complement of $n$ points. Let $f=\prod f_i ^{\lambda _i }$ for some $\lambda _i \in \Z $, so that $\log (f)=\sum \lambda _i \log (f_i )$. Given $z\in X(\overline{\Z }_p )$, choose $z_0 \in X(\Z _p ^{\mathrm{nr}})$ with the same image $\overline{z}$ in $X(\overline{\F }_p )$. Then
\[
\log (f)(z)=\log (f)(z_0 )+\int ^{z}_{z_0 }\omega ,
\]
where $\omega =\sum \lambda _i df_i /f_i $. The zeroes of $\log (f)$ on the disc at $z$ admit a simple description in terms of the reduction modulo $p$ of the differentials $df_i /f_i $. This follows from writing
\begin{align}\label{eqn:slp}
\log (f)(z) & =\log (f(z_0 ))+\log (f(z)/f(z_0 ))  \nonumber \\
& = \log (f(z_0 ))+\sum _{m>0}(-1)^{m+1}\frac{(f(z)-f(z_0 ))^m }{mf(z_0 )^m }
\end{align}
\begin{lemma}\label{lemma:slop}
Let $k =\ord _{\overline{z}}(\frac{df}{f})+1.$ Suppose $k<p$.
\begin{enumerate}
\item If $\val (\log (f(z_0 )))>0$, then the negative slopes $\lambda$ of the Newton polygon of $\log (f)$ to the right of $k$ are of the form $\frac{1}{k(p^n-p^{n-1})}$, for $n>0$. The corresponding end vertices are $(kp^n,-n)$ for $n\geq 0$.
\item If $\val (\log (f(z_0 )))=v\leq 0$, the negative slopes $\lambda $ of the Newton polygon of $\log (f)$ are $\frac{1}{k p^{v+1}}$ and $\frac{1}{k(p^{n+1}-p^{n})}$ for $n>v$. The corresponding end vertices are $(0,0)$ and $(kp^n ,-n)$ for $n>v$.
\item If $\log (f)$ contains a non-integral negative slope $>\frac{1}{p-1}$, then $\frac{df}{f}$ vanishes at $\overline{z}$ (i.e. $k>1$).
\end{enumerate}
\end{lemma}
\begin{proof}
Since $k<p$, it follows from \eqref{eqn:slp} that the first coefficient of the Taylor expansion of $\frac{f(z)-f(z_0 )}{f(z_0 )}$ which has valuation zero is the $k $th one. Hence, for all $i\geq 0$, the first coefficient of the Taylor expansion of $\log (f)$ which has valuation $-i$ is either the zeroth coefficient (if $\val (\log (f(z_0 )))=-i$) or the $kp^i$th coefficient. This implies part (1), since if $\val (\log (f(z_0 )))>0$ then the convex polygon formed by the points $(kp^i ,-i)$ is equal to the part of the Newton polygon of $\log (f)$ to the left of $x=k$. For part (2), this similarly gives the Newton polygon. 

For part (3), the preceding discussion implies that if $\log (f)$ contains a negative slope greater than $\frac{1}{p-1}$, then it arises from the vanishing mod $p$ of the first coefficient of the Taylor expansion. If the slope is non-integral, then (since the valuations of the Taylor coefficients are all integers and the $i$th coefficient has valuation at least $-\log _p (i)$) it must also be the case that the second coefficient of the Taylor expansion vanishes mod $p$, i.e. $\frac{df}{f}$ vanishs at $\overline{z}$.
\end{proof}
We now apply this to bound the ramification of points in $X(\overline{\Q }_p )\cap (\mathbb{G}_m ^3 )^{[2]}$. First we need the following Lemma, which allows us to restrict to the case of \textit{integral} points on $X$.

\begin{lemma}\label{lemma:integral}
\begin{enumerate}
\item Let $X$ be a geometrically irreducible curve over $\Z _p ^{\nr}$ with a map $(f_1 ,\ldots ,f_n ):X\to \mathbb{G}_m ^n $. Let $C$ be the projective closure of the normalisation $\widetilde{X}$ of $X$ and $\{ P_1 ,\ldots ,P_k \}=C-\widetilde{X}$ be the union of the divisors of the $f_i$. Suppose the $P_i$ are pairwise disjoint modulo $p$ (i.e. the divisor $\sum P_i $ is \'etale over $\Z _p ^{\nr}$). Then there are are finitely many surjective homomorphisms
\[
\theta _j :\mathbb{G}_m ^n \to \mathbb{G}_m ^{n-1}
\]
such that if $z\in X(\Q _p ^{\mathrm{nr}})\cap (\mathbb{G}_m ^n )^{[2]}$, and $f_i (z)$ is not in $R^\times $, then the image of $z$ in $(\mathbb{G}_m )^{n-1}$ under one of the $\theta _j$ lies in $(\mathbb{G}_m ^{n-1})^{[2]}(R)$. In particular, when $n=3$ all but finitely many points of $X(\overline{\Q }_p )\cap (\mathbb{G}_m ^3 )^{[2]}$ are integral.
\item Let $X$ be a geometrically irreducible curve over the $S$-integers of a number field $K$, and let $(f_1 ,\ldots ,f_n ):X\to \mathbb{G}_m ^3$ be an $\mathcal{O}_{K,S}$-morphism. Then there is a finite set $T\supset S$ of primes of $K$ such that all but finitely many $\overline{K}$-points of $X(\overline{K})\cap (\mathbb{G}_m ^3 )^{[2]}$ are $T$-integral.
\end{enumerate}
\end{lemma}
\begin{proof}
First we prove part (1). Write $\divv (f_i )=\sum a_{ij}P_j $. 
Suppose we have linearly independent $(n_i)$ and $(m_i )$ in $\Z ^n$ such that
\[
\prod f_i (z)^{n_i }=\prod f_i (z)^{m_i }=1.
\]
Suppose the valuation of one of the $f_i (z)$ is not zero. Our assumptions imply that there is a unique $P_j$ such that $z$ is congruent to $P_j$. Hence
\[
\sum n_i a_{ij}=\sum m_i a_{ij}=0.
\]
Hence we may take $\theta _j$ to be the map
\[
(z_1 ,\ldots ,z_n )\mapsto (z_1 ^{a_{\ell j}}z_\ell ^{-a_{1j}},\ldots ,z_n ^{a_{\ell j}}z_\ell ^{-a_{nj}}),
\]
where $a_{\ell j}\neq 0$. 

When $n=3$, we deduce that non-integral points of $X(\overline{\Q }_p )\cap (\mathbb{G}_m ^3 )^{[2]}$ map into torsion points of $\mathbb{G}_m ^2 $ under one of the $\theta _j $. Since $X$ is not contained in a translate of a proper subgroup of $\mathbb{G}_m ^3 $, the image of $X$ under $\theta _j$ is not contained in a translate of a proper subgroup. By Ihara--Lang--Serre--Tate \cite{lang}, only finitely many $\overline{\Q }_p $-points of $X$ can map to torsion points of $\mathbb{G}_m ^2 $, which completes the proof of (1).

For part (2), let $C$ again be the projective closure of the normalisation $\widetilde{X}$ of $X_K$ and $\{ P_1 ,\ldots ,P_k \} =(C-\widetilde{X})(\overline{K})$. Let $\mathcal{C}$ be a regular model of $C$ over $\mathcal{O}_K$. Then there is a finite set of primes $T$ such that $K|\Q $ is unramified at all primes not in $T$, the $P_i$ are pairwise distinct modulo $\mathfrak{p}$ for all primes $\mathfrak{p}$ not in $T$, and such that a $K$-point of $X$ is $T$-integral if and only if it comes from an $\mathcal{O}_{K,T}$-point of $\mathcal{C}-\{ P_1 ,\ldots ,P_k \}$.

As in part (1), we may construct a finite set of finite morphisms $\theta _j :X\to \mathbb{G}_m ^2 $ such that elements of $X(\overline{K})\cap \mathbb{G}_m ^{[3]}$ which are not $T$-integral map to torsion points of $\mathbb{G}_m ^2 $ under one of the $\theta _j$. As in part (1), the Ihara--Lang--Serre--Tate theorem completes the proof of part (2).
\end{proof}

\begin{proposition}\label{prop:ram}
Let $X$ be a smooth projective curve over $\Z _p ^{\nr }$ of genus $g$ and $f_1 ,\ldots ,f_n$ functions with pairwise disjoint divisors modulo $p$. Let $Z$ be the union of all points where one of the $f_i$ has a zero or pole, and let $d$ be the degree of the divisor $Z$. If $p\geq 2g+d$ and the $\overline{\Q }_p $-points of $Z$ have pairwise distinct reductions in $\overline{\F }_p$, then every point of $X(\overline{\Q }_p )\cap (\mathbb{G}_m ^{n})^{[2]}$ has ramification degree at most $2g+d$.
\end{proposition}
\begin{proof}
By Lemma \ref{lemma:integral}, we may reduce to the case where all points on $X(\overline {\Q }_p )\cap (\mathbb{G}_m ^{r})^{[2]}$ are integral with respect to the $f_i$. Let $z$ be a ramified point in $X(\overline{\Q }_p )\cap (\mathbb{G}_m ^r )^{[2]}$. Let $\overline{z}$ be the reduction of $z$ modulo $p$. We may choose a point $z_0 \in X(\Q _p ^{\mathrm{nr}})$ congruent to $z$ modulo $p$, and an integral parameter $T$ at $z_0$, such that $\val (T(z))$ is not in $\Z $.

Recall that every codimension $d$ connected subgroup of $\mathbb{G}_m ^n $ is the connected component of the kernel of a surjection $\mathbb{G}_m ^n \to \mathbb{G}_m ^d $. Let $g_1 $ and $g_2 $ be composite maps
\[
X\to \mathbb{G}_m ^r \to \mathbb{G}_m ^2
\]
such that $g_i (z)$ is torsion. If $]\overline{z}[$ contains a point in $(\mathbb{G}_m ^n )^{[2]}$ of ramification degree $N$, then it contains at least $N$ such points, since $X(\overline{\Q }_p )\cap (\mathbb{G}_m ^n )^{[2]}$ is stable under the action of $\Gal (\overline{\Q }_p |\Q _p ^{\nr})$. Hence it is enough to prove that the number of common slopes of $g_1 $ and $g_2$, counted with multiplicity, is less than $2g+d$. By the description of the slopes of $\log (g_i )$ in Lemma \ref{lemma:slop}, it is enough to prove that, for any $\lambda <\frac{1}{2g+d}$, there exist $\mu _1 ,\mu _2 \in \Z$ such that $\lambda $ is not a slope of $\mu _1 \log (g_1 )+\mu _2 \log (g_2 )$.

Since all positive slopes less than $\frac{1}{2g+d}$ are of the form $\frac{1}{kp^i}$ or $\frac{1}{k(p^i -p^{i-1})}$, and $k<2g+d<p$, it is enough to first prove the result for $\lambda $ of the form $\frac{1}{k(p^i-p^{i-1})}$ and then for $\lambda $ of the form
Without loss of generality we may assume $\ord _{\overline{z}}\frac{dg_1 }{g_1 }\leq \ord _{\overline{z}}\frac{dg_2 }{g_2 }$. Define $\omega _1 =\dlog (g_1 )$ and $\omega _2 =\dlog (g_2 )-\alpha \dlog (g_1 )$, where $\alpha$ is the value of $\frac{\dlog g_2 }{\dlog (g_1 )}$ at $\overline{z}$. Then 
\[
\int ^z \omega _1 =\int ^z \omega _2 =0.
\]
Let $k_i $ be the order of vanishing of $\omega _i$ at $\overline{z}$, and let $a_i$ be the valuation of $\int ^{z_0 }_b \omega _i $. Then $k_1 <k_2 \leq 2g+d$. On the other hand by Lemma \ref{lemma:slop}, the slopes of $\int \omega _i $, viewed as power series in $T$, are $\frac{1}{k_i p^{a_i +1}}$ and/or $\frac{1}{k_i (p^n-p^{n-1})}$. Since $p>k_2 >k_1 $, $\int \omega _1 $ and $\int \omega _2 $ do not have a common (negative) slope less than $\frac{1}{2g+d}$ of the form $\frac{1}{k_i (p^n-p^{n-1})}$. 

We now consider slopes of the form $\frac{1}{kp^i}$. Let $v_i$ be the valuation of $\log (g_i (z_0 ))$. If $v_1 =v_2$, we may similarly we may choose $\mu _1 ,\mu _2 \in \Z -p\Z $ such that 
\[
\val (\mu _1 \log (g_1 (z_0 ))+\mu _2 \log (g_2 (z_0 )))>v_1 .
\]
It follows that if $\log (g_1 )$ and $\log (g_2 )$ both have a slope of the form $\frac{1}{k_i p^{v_i }}$, then this is not a slope of $\mu _1 \log (g_1 )+\mu _2 \log (g_2 )$.
\end{proof}
\begin{remark}
At first sight, this proof depends in an essential way on the explicit description of the formal expansion of $\log $ around $1$. The beautiful insight of Coleman (see \cite{coleman:ramified}, \cite{DP}) is that there is another expansion, valid for \textit{any} abelian integral $\int \omega $, which can similarly capture the slopes of $\int \omega $ in terms of zeroes of differentials on the special fibre. We will not use this in what follows.
\end{remark}

\begin{lemma}\label{lemma:fin}
Let $X$ be as in Proposition \ref{prop:ram}. Then only finitely many residue disks contain a point in $X\cap (\mathbb{G}_m ^r )^{[2]}$ that is ramified at $p$.
\end{lemma}
\begin{proof}
By Lemma \ref{lemma:slop}, if $z$ is a ramified point and $z$ satisfies 
\[
\prod f_i (z)^{n_i }\in \mu _{\infty },
\]
then the differential form
$
\sum n_i \frac{df_i }{f_i }
$ 
vanishes at $\overline{z}$, the reduction of $z$ modulo the maximal ideal of $\overline{\Z }_p $. We may assume that the GCD of the $n_i $ is $1$. Since the vanishing condition only depends on the $n_i$ modulo $p$, the number of residue discs that contain a point in $X\cap (\mathbb{G}_m ^n )^{[2]}$ that is ramified at $p$ is equal to the number of residue disks where $\sum n_i \frac{df_i }{f_i }$ vanishes as $(n_i )$ ranges over $\{ 0,\ldots ,p-1\}^n -(0,\ldots ,0)$.
\end{proof}
\begin{remark}
In fact, it follows from the proof that the number of residue disks is bounded by $\frac{p^n-1}{p-1}(2g-2+d-1)$.
\end{remark}
Lemma \ref{lemma:fin} and Proposition \ref{prop:ram} should be viewed as weaker versions of the following theorem of Coleman \cite{coleman:ramified} in the context of Manin--Mumford.
\begin{theorem}[Coleman]\label{thm:CRMM}
If $X/\Q _p ^{\mathrm{nr}}$ is a smooth projective curve of genus $g\leq p/2$, with good reduction, and 
\[
\iota :X\to \Jac (X)
\]
is an Abel--Jacobi morphism, then 
\[
\iota (X(\overline{\Q }_p ))\cap \Jac (X)_{\tors }\subset \iota (X(\Q _p ^{\mathrm{nr}})).
\]
\end{theorem}
In fact, in \cite{DP} it is shown that 
\[
\iota (X(\overline{\Q }_p ))\cap \Jac (X)(\Q _p ^{\mathrm{nr}})_{\divv }\subset \iota (X(\Q _p ^{\mathrm{nr}})),
\]
where $\Jac (X)(\Q _p ^{\mathrm{nr}})_{\divv }\subset \Jac (X)(\overline{\Q }_p )$ denotes the divisible hull of $\Jac (X)(\Q _p ^{\mathrm{nr}})$.
\section{The Buium step: bounding anomalous residue discs}
In this section we aim to prove the following result. Let $X/\Z_p ^{\mathrm{nr}}$ be a smooth geometrically irreducible curve, and let $(f_1 ,f_2 ,f_3 ):X\hookrightarrow \mathbb{G}_m ^3 $ be a closed immersion.
\begin{proposition}\label{prop:voloch}
Suppose that $X$ is not contained in a coset of a proper algebraic subgroup of $\mathbb{G}_m ^3 $ modulo $p$, and suppose that the divisors of the $f_i$ are linearly independent in $\Div ^0 (X)\otimes \F _p $. Then there are finitely many residue discs containing an unramified point in $X(\Q _p ^{\mathrm{nr}})\cap (\mathbb{G}_m ^n )^{[2]}$.
\end{proposition}
The proof closely follows a related result of Voloch \cite{voloch} which gives a quantitative bound on the proximity of plane curves to roots of unity, strengthening a (special case of) the Tate--Voloch theorem on the proximity of algebraic varieties to roots of unity \cite{TV}. A more conceptual approach to the proof given is provided by the theory of arithmetic jet spaces \cite{BPS},\cite{buium:96}, however because we are working in $\mathbb{G}_m $ their role will be hidden and we can give a more elementary exposition. The reader may wish to compare our proof with the proof of the following theorem of Buium in the context of Manin--Mumford \cite{buium:96}.
\begin{theorem}[Buium]\label{thm:BMM}
Let $p$ be an odd prime, and $X/\Z _p ^{\mathrm{nr}}$ be a smooth projective curve of genus $g>1$, and $\iota :X\to \Jac (X)$ an Abel--Jacobi morphism relative to a point $b\in X(\Z _p ^{\mathrm{nr}})$. Then the image of $X(\Z _p ^{\mathrm{nr}})\cap \iota ^{-1}(\Jac (X)_{\tors})$ in $X(\overline{\F }_p )$ has size at most 
\[
p^{4g}3^g [p(2g-2)+6g]g!.
\]
\end{theorem}

We now give the proof of Proposition \ref{prop:voloch}. Let $X$ be as in Proposition \ref{prop:voloch}. Suppose $z$ is in $X(\Z_p ^{\nr })$, and $\prod f_i (z)^{n_i }$ and $\prod f_i (z)^{m_i }$ are roots of unity. We may (and do) assume that the GCD$(n_i ) =$GCD$(m_i )=1$. Let $W_1 (\overline{\F }_p )=W(\overline{\F }_p )/p^2 $ denote the truncated $p$-typical Witt vectors of $\overline{\F }_p $. To prove Proposition \ref{prop:voloch}, it is enough to show that there are only finitely many $W_1 (\overline{\F }_p )$ points of $X$ such that, for some $(n_i )$ and $(m_i )$ both of GCD 1, $\prod f_i (z)^{n_i }$  and $\prod f_i (z)^{m_i }$ are in the image $D$ of $\mu _{\infty }(\Z _p ^{\mathrm{nr}})$ in $W_1 (\overline{\F }_p )^\times $. Since taking $p$th powers maps $W_1 (\overline{\F }_p )$ into $D$, we may reduce to the case that all the $n_i$ and $m_i $ are in the interval $\{ 0,\ldots ,p-1 \}$.

Hence to prove Proposition \ref{prop:voloch}, we may replace an \textit{infinite} collection of rank one tori with a \textit{finite} collection, namely the subtori of the form
\[
\prod x_i ^{n_i }=\prod x_i ^{m_i }\in \mu _{\infty }
\]
for $(n_i )$ and $(m_i )$ in $\{ 0,\ldots ,p-1 \}^3 $, linearly independent modulo $p$. It is then enough to prove the following proposition

\begin{proposition}\label{prop:voloch2}
Let $f$ and $g$ be functions on $X$ such that the divisors of $f$ and $g$ are not multiples of $p$, and $\dlog (f)$ and $\dlog (g)$ are $\F _p $-linearly independent in $\Omega _{\overline{\F }_p (X)|\overline{\F }_p }$. Then there are only finitely many points $z$ in $X(W_1 (k))$ such that $f(z)$ and $g(z)$ are both in $D$.
\end{proposition}
\begin{proof}
We will closely follow the strategy of Voloch \cite{voloch}, who proved an explicit bound on the number of points in $C(W_1 (\overline{\F }_p ))\cap D\times D$, where $C$ is a plane curve of degree less than $\sqrt{p-2}$.

Let $H\in \Z_p ^{\mathrm{nr}} [x,y]$ be minimal such that $H(f,g)=0$. It is enough to show that the intersection of the $W_1 (\overline{\F }_p )$-zeroes of $H$ with the roots of unity $D\times D$ is finite. Let $h\in \overline{\F }_p [x,y]$ be the reduction of $H$. As explained in \cite{voloch}, it is enough to show that $h(x,y)$ does not divide $G(x,y)$ in $\overline{\F }_p [x,y]$, where $G(x,y)$ is the reduction modulo $p$ of $(H^{\sigma }(x^p,y^p)-H(x,y)^p)/p$, and $\sigma $ is the lift of the Frobenius morphism from $\overline{\F }_p $ to $W(\overline{\F }_p )$

If $h$ divides $G$, then $dG$ vanishes on the zero locus of $h$. We have
\begin{align*}
dG & =(x^{p-1} \frac{\partial h}{\partial x}(x^p ,y^p )-h^{p-1}\frac{\partial h}{\partial x})dx \\
& +(y^{p-1} \frac{\partial h}{\partial y}(x^p ,y^p )-h^{p-1}\frac{\partial h}{\partial y})dy. 
\end{align*} 
We deduce that if $h$ divides $G$ then
\[
x^{p-1} \frac{\partial h}{\partial x}(x^p ,y^p )\frac{\partial h}{\partial y}=y^{p-1} \frac{\partial h}{\partial y}(x^p ,y^p )\frac{\partial h}{\partial x}
\]
on the zero locus of $h$.
Since the divisors of $f$ and $g$ are not multiples of $p$, $\frac{\partial h}{\partial x}$ and $\frac{\partial h}{\partial y}$ are both nonzero modulo $p$. Hence this further simplifies to
\[
x^{p-1} \frac{\partial h}{\partial x}^{p-1}=y^{p-1} \frac{\partial h}{\partial y}^{p-1}
\]
on the zero locus of $h$. Finally it is enough to show that there are not constants $a,b\in \overline{\F }_p$ such that
\[
x\frac{\partial h}{\partial x}+ay\frac{\partial h}{\partial y}+bh=0.
\]
Note that such a relation implies 
\[
\frac{dg}{g}-a\frac{df}{f}\equiv 0 \mod p
\]
on $X$, which contradicts our assumption of linear independence of the logarithmic differentials.
\end{proof}

\section{$p$-adic integration and functional transcendence}
To conclude the proof of Theorem \ref{thm:main}, we need to study functional transcendence properties of Coleman integrals on higher dimensional varieties, following \cite{d:unlikely} and \cite{hast:ax}. Let $T$ be a $g$-dimensional torus over a field $K$ of characteristic zero. Let $L$ be the Lie algebra of $T$. Let $\widehat{T}$ and $\widehat{L}$ be the associated formal groups and let 
\[
\Delta \subset \widehat{T}\times \widehat{L}
\]
be the graph of the formal logarithm.
\begin{theorem}[Ax, \cite{ax1971schanuel}]\label{thm:AS}
Let $V\subset T$ and $W\subset L$ be subvarieties.
If
\[
\dim (W)>\dim V-g,
\]
then $\pi (\Delta \cap \widehat{V}\times \widehat{W})$ is contained in the formal completion of a proper subgroup of $T$, where
\[
\pi :T\times L\to T
\]
is the projection.
\end{theorem}
\begin{remark}
For the dictionary between Ax's statement of the theorem and the one we give above, see \cite{tsimerman:unlikely}.
\end{remark}

This has the following corollary.
\begin{proposition}\label{prop:AStoC}
Let $X$ be as in Theorem \ref{thm:main}. Fix $(\overline{z}_1 ,\overline{z}_2 )$ in $X(\overline{\F }_p )^2 $. Let $D$ be a closed polydisc inside $](\overline{z}_1 ,\overline{z}_2 )[$. Then the set $S$ of points $(z_1 ,z_2 )$ in $D(\overline{\Q }_p )$ such that there are $n_1 ,n_2 $ in $\Z$, not both zero, such that $f_i (z_1 )^{n_1 }=f_i (z_2 )^{n_2 }$ for all $i$, is not Zariski dense in $X^2 $. Furthermore each irreducible component of the Zariski closure is contained in a translate of a subtorus of the form 
\[
\{ (x_i ,y_i )\in \mathbb{G}_m ^6 :x_i ^{a }=  y_i ^{b} \}
\]
for $a,b \in \Z ^2 -\{ (0,0)\}$.
\end{proposition}
\begin{proof}
$S$ is contained in the zero locus $Z$ of $\log f_i (z_1 )\log f_j (z_2 )-\log f_j (z_1 )\log f_i (z_2 )$. Since $\log (f_i )|_D$ is a rigid analytic function on a closed polydisc, to prove that $Z$ is not Zariski dense, it is enough to prove it for each irreducible component, and for this it is enough to prove it at the formal completion of a point $(z_1 ,z_2 )$ lying in $S$. View $X^2$ as a subvariety $Y$ of $\mathbb{G}_m ^6$ via the maps
\[
(w_1 ,w_2 )\mapsto (f_i (w_1 )/f_i (z_1 ),f_i (w_2 )/f_i (z_2 )).
\]
Let $W\subset \mathbb{A}^6$ be the codimension two subvariety of $3$ by $2$ matrices $M_{ij}$ such that $\rk (M_{ij}+f_i (z_j ))\leq 1$. Then the formal completion of $Z$ at $(z_1 ,z_2 )$ is isomorphic to $\pi (\Delta \cap (\widehat{Y}\times \widehat{W}))$, and the proposition follows from Theorem \ref{thm:AS}.
\end{proof}

\section{Completion of proof}
We first explain how the results of Buium and Coleman complete the proof of Manin--Mumford. Let $X$ and $p$ be as in Theorems \ref{thm:CRMM} and \ref{thm:BMM}. By Theorem \ref{thm:CRMM}, all torsion points are unramified. By Theorem \ref{thm:BMM}, only finitely many points of $X(\overline{\F }_p )$ lift to a torsion point in $X(\Q _p ^{\mathrm{nr}})$. Hence it is enough to prove that there are only finitely many torsion points in $]\overline{z}[(\Q _p ^{\mathrm{nr}})$. But this follows from standard properties of torsion points on abelian varieties, or (more in the spirit of the paper) by elementary properties of Coleman integrals -- the function $\int _b \omega $ has at most one unramified zero in $]\overline{z}[$ if $\omega $ is a differential which does not vanish at $\overline{z}$.

We now complete the proof of Theorem \ref{thm:main}, following an analogous strategy. It is enough to prove finiteness of $X(\overline{\Q }) \cap (\mathbb{G}_m ^3 )^{[2]}\cap ]\overline{z}[(\overline{\Q }_p )$, where $\overline{z}$ is a fixed point in $X(\overline{\F }_p )$. Indeed, by Lemma \ref{lemma:fin}, the image of ramified points of $X(\overline{\Q }_p )\cap (\mathbb{G}_m ^3 )^{[2]}$ in $X(\overline{\F }_p )$ is finite, and by Proposition \ref{prop:voloch} the image of $X(\Q_p ^{\mathrm{nr}})\cap (\mathbb{G}_m ^3 )^{[2]}$ in $X(\overline{\F }_p )$ is finite. Furthermore Proposition \ref{prop:ram} implies it is enough to prove this with $]\overline{z}[$ replaced with a closed disc of radius $p^{-\frac{1}{2g+d}}$.

We will deduce this result from Siegel's theorem \cite{siegel}. That is, we will show that all points in $X(\overline{\Q })\cap (\mathbb{G}_m ^3 )^{[2]}$ are defined over a fixed finite extension $L$ of $K$. By part (2) of Lemma \ref{lemma:integral}, these points have to be integral at all but finitely many primes, and hence Siegel's theorem implies finiteness. This is analogous to the proof of Mordell--Lang in \cite{DP}, where we use the Buium--Coleman strategy to reduce Mordell--Lang to Mordell.

As in \cite{DP}, to bound the field of definition we will consider Coleman integrals on a higher dimensional space and apply Ax--Schanuel. Fix $\overline{z}_1 ,\overline{z}_2 \in X(\overline{\F }_p )$ lying in the image of $X(\overline{\Q }_p )\cap (\mathbb{G}_m ^3 )^{[2]}$. 
Let $C_1 ,\ldots ,C_n $ be the positive dimensional irreducible components of the Zariski closure of 
\[
Z=\{ (z_1 ,z_2 )\in ](\overline{z}_1 ,\overline{z}_2 )[(\overline{\Q}_p ): \rk (\log (f_i (z_j ))_{1\leq i\leq 3,1\leq j\leq 2)} \leq 1\}\subset X(\overline{\Q }_p )^2 .
\]
By Proposition \ref{prop:ram}, this is the same as 
\[
\{ (z_1 ,z_2 )\in D: \rk (\log (f_i (z_j ))_{1\leq i\leq 3,1\leq j\leq 2)} \leq 1\}\subset X(\overline{\Q }_p )^2 .
\]
where $D$ is a closed polydisc inside $](\overline{z}_1 ,\overline{z}_2 )[$. Hence by Proposition \ref{prop:AStoC}, each $C_i $ is a curve contained in a torus translate of the form
\begin{equation}\label{eqn:Ti}
T_i =\{ (x_j ,y_j )\in \mathbb{G}_m ^6 : x_j ^{a_i} =\lambda _{i} y_j ^{b_i} \}
\end{equation}
for some $\lambda _{i} $ in $\overline{\Q }$ and coprime $a_i ,b_i \in \Z$. Let $L/K $ be a finite extension containing all the $\lambda _{i}$.

Now suppose $C$ is a positive dimensional irreducible component of the Zariski closure of the intersection of $Z$ with the set 
\[
\{ (x,\sigma (x)):x\in X(\overline{\Q }),\sigma \in \Gal (\overline{L}|L) \}.
\]
Then $C$ is one of the $C_i$ above. In particular it contains infinitely many points $(x,\sigma (x))$ such that $f_j (x)^{a_i }=\lambda _i f_j (\sigma (x))^{b_i }$. We deduce that $f_j (x)$ is in the divisible hull of $\lambda _i$, i.e. in the subgroup
\[
\cup _{N>0}\{ v\in \overline{\Q }_p ^\times :v^N \in \lambda _i ^{\Z } \}.
\]
If $a_i \neq b_i$ we have reduced to Mordell--Lang for a finite number of curves inside $\mathbb{G}_m ^2 $, which is a theorem of Liardet \cite{liardet}.

Hence the only remaining case to deal with is where $a_i =b_i $ and $\lambda _i $ is a root of unity. In this case we see that there is an automorphism $\tau $ of $X$ such that $f_j (\tau (x))=\zeta f_j (x)$ for some root of unity $\zeta $. We say $X$ is \textit{minimal} if the map $X\to \mathbb{G}_m ^3 $ does not factor through a map
\begin{equation}\label{eqn:tori}
\theta =((.)^{n_1 },(.)^{n_2 },(.)^{n_3 }):\mathbb{G}_m ^3 \to \mathbb{G}_m ^3 .
\end{equation}
with some of the $n_i >1$. For any non-minimal $X$, there is map of tori $\theta $ as in \eqref{eqn:tori} such that $\theta (X)$ is minimal, and Theorem \ref{thm:main} for $\theta (X)$ implies Theorem \ref{thm:main} for $X$. The existence of an automorphism $\tau $ as above with $\zeta \neq 1 $ implies that $X$ is non-minimal. Hence we may (and do) assume that $X$ is minimal, and hence that $\zeta  =1$, reducing to the case that $C$ is diagonal, and hence to Siegel's theorem.

\bibliography{bib_ZP}

\begin{thebibliography}{CMPZ16}

\bibitem[Ax71]{ax1971schanuel}
J.~Ax.
\newblock On {S}chanuel's conjectures.
\newblock {\em Annals of mathematics}, pages 252--268, 1971.

\bibitem[BD25]{BD25}
F.~Barroero and G.~A. Dill.
\newblock Distinguished categories and the {Z}ilber--{P}ink conjecture.
\newblock {\em American Journal of Mathematics (to appear)}, 2025.

\bibitem[BMZ99]{BMZ99}
E.~Bombieri, D.~Masser, and U.~Zannier.
\newblock Intersecting a curve with algebraic subgroups of multiplicative groups.
\newblock {\em Internat. Math. Res. Notices}, (20):1119--1140, 1999.

\bibitem[BPS23]{BPS}
A.~Bertapelle, E.~Previato, and A.~Saha.
\newblock Arithmetic jet spaces.
\newblock {\em J. Algebra}, 623:127--153, 2023.

\bibitem[Bui95]{buium:95}
A.~Buium.
\newblock Differential characters of abelian varieties over {$p$}-adic fields.
\newblock {\em Invent. Math.}, 122(2):309--340, 1995.

\bibitem[Bui96]{buium:96}
A.~Buium.
\newblock Geometry of {$p$}-jets.
\newblock {\em Duke Math. J.}, 82(2):349--367, 1996.

\bibitem[Bui97]{buium:97}
A.~Buium.
\newblock Differential characters and characteristic polynomial of {F}robenius.
\newblock {\em J. Reine Angew. Math.}, 485:209--219, 1997.

\bibitem[CMPZ16]{CMPZ}
L.~Capuano, D.~Masser, J.~Pila, and U.~Zannier.
\newblock Rational points on {G}rassmannians and unlikely intersections in tori.
\newblock {\em Bull. Lond. Math. Soc.}, 48(1):141--154, 2016.

\bibitem[Col85a]{coleman:chabauty}
R.~F. Coleman.
\newblock Effective {C}habauty.
\newblock {\em Duke Math. J.}, 52(3):765--770, 1985.

\bibitem[Col85b]{coleman:torsion}
R.~F. Coleman.
\newblock Torsion points on curves and {$p$}-adic abelian integrals.
\newblock {\em Ann. of Math. (2)}, 121(1):111--168, 1985.

\bibitem[Col87]{coleman:ramified}
R.~F. Coleman.
\newblock Ramified torsion points on curves.
\newblock {\em Duke Math. J.}, 54(2):615--640, 1987.

\bibitem[Dog22]{D22}
N.~Dogra.
\newblock p-adic integrals and linearly dependent points on families of curves {I}.
\newblock {\em arXiv preprint arXiv:2206.04304}, 2022.

\bibitem[Dog24]{d:unlikely}
N.~Dogra.
\newblock Unlikely intersections and the {C}habauty-{K}im method over number fields.
\newblock {\em Math. Ann.}, 389(1):1--62, 2024.

\bibitem[DP25]{DP}
N.~Dogra and S.~Pandit.
\newblock A {B}uium--{C}oleman bound for the {M}ordell--{L}ang conjecture.
\newblock {\em arXiv preprint arXiv:2504.10155}, 2025.

\bibitem[Fal83]{faltings83}
G.~Faltings.
\newblock Endlichkeitss{\"a}tze f{\"u}r abelsche variet{\"a}ten {\"u}ber zahlk{\"o}rpern.
\newblock {\em Inventiones mathematicae}, 73(3):349--366, 1983.

\bibitem[Fal94]{faltings:general}
G.~Faltings.
\newblock The general case of {S}. {L}ang's conjecture.
\newblock In {\em Barsotti {S}ymposium in {A}lgebraic {G}eometry ({A}bano {T}erme, 1991)}, volume~15 of {\em Perspect. Math.}, pages 175--182. Academic Press, San Diego, CA, 1994.

\bibitem[FvdP04]{FvdP}
J.~Fresnel and M.~van~der Put.
\newblock {\em Rigid analytic geometry and its applications}, volume 218 of {\em Progress in Mathematics}.
\newblock Birkh\"{a}user Boston, Inc., Boston, MA, 2004.

\bibitem[Has21]{hast:ax}
D.~R. Hast.
\newblock Functional transcendence for the unipotent {A}lbanese map.
\newblock {\em Algebra Number Theory}, 15(6):1565--1580, 2021.

\bibitem[Kob84]{koblitz}
N.~Koblitz.
\newblock {\em {$p$}-adic numbers, {$p$}-adic analysis, and zeta-functions}, volume~58 of {\em Graduate Texts in Mathematics}.
\newblock Springer-Verlag, New York, second edition, 1984.

\bibitem[Lan65]{lang}
S.~Lang.
\newblock Division points on curves.
\newblock {\em Ann. Mat. Pura Appl. (4)}, 70:229--234, 1965.

\bibitem[Lia74]{liardet}
P.~Liardet.
\newblock Sur une conjecture de {S}erge {L}ang.
\newblock {\em C. R. Acad. Sci. Paris S\'er. A}, 279:435--437, 1974.

\bibitem[LS07]{LS}
B.~Le~Stum.
\newblock {\em Rigid cohomology}, volume 172 of {\em Cambridge Tracts in Mathematics}.
\newblock Cambridge University Press, Cambridge, 2007.

\bibitem[Mau08]{maurin}
G.~Maurin.
\newblock Courbes alg\'{e}briques et \'{e}quations multiplicatives.
\newblock {\em Math. Ann.}, 341(4):789--824, 2008.

\bibitem[Pil22]{pila}
J.~Pila.
\newblock {\em Point-counting and the {Z}ilber-{P}ink conjecture}, volume 228 of {\em Cambridge Tracts in Mathematics}.
\newblock Cambridge University Press, Cambridge, 2022.

\bibitem[Pin05]{pink}
R.~Pink.
\newblock A combination of the conjectures of {M}ordell-{L}ang and {A}ndr\'{e}-{O}ort.
\newblock In {\em Geometric methods in algebra and number theory}, volume 235 of {\em Progr. Math.}, pages 251--282. Birkh\"{a}user Boston, Boston, MA, 2005.

\bibitem[PW06]{PZ}
J.~Pila and A.~J. Wilkie.
\newblock The rational points of a definable set.
\newblock {\em Duke Math. J.}, 133(3):591--616, 2006.

\bibitem[Ray83]{Ray83}
M.~Raynaud.
\newblock Courbes sur une vari\'{e}t\'{e} ab\'{e}lienne et points de torsion.
\newblock {\em Invent. Math.}, 71(1):207--233, 1983.

\bibitem[Sie]{siegel}
C.~L. Siegel.
\newblock \"{U}ber einige {A}nwendungen diophantischer {A}pproximationen [reprint of {A}bhandlungen der {P}reuss ischen {A}kademie der {W}issenschaften. {P}hysikalisch-mathematische {K}lasse 1929, {N}r. 1].
\newblock In {\em On some applications of {D}iophantine approximations}, Quad./Monogr.

\bibitem[Sto19]{stoll:uniform}
M.~Stoll.
\newblock Uniform bounds for the number of rational points on hyperelliptic curves of small {M}ordell-{W}eil rank.
\newblock {\em J. Eur. Math. Soc. (JEMS)}, 21(3):923--956, 2019.

\bibitem[Tsi15]{tsimerman:unlikely}
J.~Tsimerman.
\newblock Ax-schanuel and o-minimality.
\newblock {\em O-Minimality and Diophantine Geometry}, 421:216, 2015.

\bibitem[TV96]{TV}
J.~Tate and J.~F. Voloch.
\newblock Linear forms in {$p$}-adic roots of unity.
\newblock {\em Internat. Math. Res. Notices}, (12):589--601, 1996.

\bibitem[Tze00]{tzermias}
P.~Tzermias.
\newblock The {M}anin-{M}umford conjecture: a brief survey.
\newblock {\em Bull. London Math. Soc.}, 32(6):641--652, 2000.

\bibitem[Vol99]{voloch}
J.~F. Voloch.
\newblock Plane curves and {$p$}-adic roots of unity.
\newblock {\em Bull. Austral. Math. Soc.}, 60(3):479--482, 1999.

\bibitem[Zan12]{Zann}
U.~Zannier.
\newblock {\em Some problems of unlikely intersections in arithmetic and geometry}, volume 181 of {\em Annals of Mathematics Studies}.
\newblock Princeton University Press, Princeton, NJ, 2012.
\newblock With appendixes by David Masser.

\bibitem[Zar96]{zarhin}
Yu.~G. Zarhin.
\newblock {$p$}-adic abelian integrals and commutative {L}ie groups.
\newblock volume~81, pages 2744--2750. 1996.
\newblock Algebraic geometry, 4.

\bibitem[Zil02]{zilber}
B.~Zilber.
\newblock Exponential sums equations and the {S}chanuel conjecture.
\newblock {\em J. London Math. Soc. (2)}, 65(1):27--44, 2002.

\end{thebibliography}
\bibliographystyle{alpha}

\end{document}